\documentclass[12pt]{amsart}

\title{A topological fibrewise fundamental groupoid}
\author{David Michael Roberts}
\email{david.roberts@adelaide.edu.au}
\address{School of Mathematical Sciences\\
University of Adelaide\\
SA 5005\\
Australia}
\thanks{Supported by the Australian Research Council 
(grant number DP120100106). This document is released under a CC0 license \url{http://creativecommons.org/publicdomain/zero/1.0/}}

\date{}

\usepackage{amsmath}
\usepackage{amssymb}
\usepackage{graphicx}
\usepackage{epsfig}
\usepackage[mathscr]{eucal} 
\usepackage{mathrsfs}
\usepackage{amsthm}
\usepackage{url}

\usepackage{ifpdf}
\ifpdf
\usepackage[all,pdf]{xy} 
\else
\input xy
\xyoption{all}
\xyoption{2cell}
\xyoption{v2}
\fi

\newtheorem{theorem}{Theorem}[section]
\newtheorem{lemma}[theorem]{Lemma}
\newtheorem{proposition}[theorem]{Proposition}
\newtheorem{corollary}[theorem]{Corollary}
\theoremstyle{definition}
\newtheorem{definition}[theorem]{Definition}
\newtheorem{example}[theorem]{Example}

\newcommand{\pionevert}{\Pi_1^{vert}}
\newcommand{\Top}{\mathbf{Top}}
\newcommand{\TopGpd}{\mathbf{TopGpd}}
\newcommand{\Gpd}{\mathbf{Gpd}}

\newcommand*\cocolon{%
        \nobreak
        \mskip6mu plus1mu
        \mathpunct{}%
        \nonscript
        \mkern-\thinmuskip
        {:}%
        \mskip2mu
        \relax
}

\DeclareMathOperator{\disc}{disc}
\DeclareMathOperator{\codisc}{codisc}
\DeclareMathOperator{\Obj}{Obj}
\DeclareMathOperator{\Mor}{Mor}
\DeclareMathOperator{\id}{id}
\DeclareMathOperator{\pr}{pr}
\DeclareMathOperator{\ev}{ev}

\begin{document}


\begin{abstract}

	It is well-known that for certain local connectivity assumptions the fundamental groupoid of a topological space can be equipped with a topology making it a topological groupoid.
	In other words, the fundamental groupoid functor can be lifted through the forgetful functor from topological groupoids to groupoids.
	This article shows that for a map $Y \to X$ with certain relative local connectivity assumptions, the \emph{fibrewise} fundamental groupoid can also be lifted to a topological groupoid over the space $X$.
	This allows the construction of a simply-connected covering space in the setting of fibrewise topology, assuming a local analogue of the definition of an ex-space.
	When applied to maps which are up-to-homotopy locally trivial fibrations the result is a categorified version of a covering space. 
	The fibrewise fundamental groupoid can also be used to define a topological fundamental bigroupoid of a (suitably locally connected) topological space.

\end{abstract}

\maketitle

\section{Introduction}

The fundamental groupoid of a semilocally 1-connected space $Y$ can be made into a topological groupoid, with object space $Y$ (\cite{Danesh-Naruie}, published in \cite{Brown-Danesh-Naruie_75}).  
Given a family of spaces, that is, a map $p\colon Y\to X$, we can define a \emph{vertical} fundamental groupoid $\pionevert(p)$, which is a family of groupoids parameterised by the underlying set of $X$. 
The groupoid at the point $x\in X$ is the fundamental groupoid $\Pi_1(Y_x)$ of the fibre $Y_x$ over $x$. 
It is a natural question to ask whether one can put an appropriate topology on the arrow set of $\pionevert(Y\to X)$ making it a topological groupoid over the space $X$.

Given local connectivity assumptions, we define such a topology, giving a functor
\[
  \pionevert{} \colon \Top/X \to \TopGpd/X
\]
which is compatible with pullback along maps $Z\to X$. This functor is additionally a \emph{simplicial} functor for the obvious simplicial enrichments of its domain and codomain. 
This generalises a construction given in the author's thesis \cite{Roberts_phd}, which assumed $p$ was a fibre bundle.

This idea of performing constructions \emph{relative to}, or \emph{parameterised by}, a base space has been common in algebraic geometry for decades.
In homotopy theory itself, some work has been done  on \emph{ex-spaces} (maps with a specified section with closed image), starting with \cite{James_69}, but uptake has not been widespread until relatively recently (see e.g.\ \cite{Crabb-James_98,May-Sigurdsson}).
The construction in this paper allows one to define a \emph{parameterised universal covering space}: that is, a continuously varying family of covering spaces of the fibres of a map $Y\to X$ satisfying certain properties.
The usual fundamental groupoid can be further refined to a Lie groupoid, and the smooth universal covering space can be extracted from it.
The minimal conditions when it is possible to perform this smooth construction in the fibrewise setting will be not be treated in the present paper.

We give two further applications of the construction given in this article. 
The first is that for fibrations which are locally trivial up to homotopy, the vertical fundamental groupoid gives a categorified bundle, or \emph{2-bundle} (e.g.\ \cite{Bartels}).
More precisely, it gives, in the terminology of \cite{Roberts_phd}, a \emph{2-covering space}. 
The theory of 2-bundles has been generalised to $n$- and $\infty$-bundles \cite{NSS_14}, but it is a nontrivial exercise to come up with examples, and this paper is a contribution in that direction at the modest level of $n=1$.
The second application is to show that the fundamental bigroupoid of a space \cite{Danny_phd,HKK_01} can be lifted to a topological bigroupoid. 
This was shown in \cite{Roberts_phd,Roberts_13b} by direct construction, but here uses the structure of the free path space as a sort of $A_\infty$-groupoid (see also \cite{Cheng_CT14}).
Given a topological bigroupoid one can use the construction in \cite{Bakovic_phd} to arrive at a 2-bundle in a different way, namely as the fibre at an object of the source functor, with the projection given by the restriction of the target functor. 
In general, though, it is not easy to supply topologically nontrivial examples of topological bigroupoids. What this paper does \emph{not} do is give an example where the hom-groupoids are not equivalent to topologically discrete groupoids.

There are other constructions that assign internal groupoid objects of a homotopical nature to maps of spaces. 
There is a construction by Brown and Janelidze \cite{Brown-Janelidze} of a double groupoid $\rho(q)$ from a map $q\colon M\to B$ of topological spaces, which in particular contains the 2-groupoid of Kamps and Porter \cite{Kamps-Porter}, defined when $q$ is a fibration.
The groupoid of vertical arrows of $\rho(q)$ is the \v Cech groupoid of the map $q$ (and in particular is an equivalence relation); the groupoid of horizontal arrows of $\rho(q)$ is just $\Pi_1(M)$. 
The vertical fundamental groupoid has identity-on-objects functors to both of these.
The functor to the vertical groupoid is analogous to the canonical identity-on-objects functor from any groupoid to a codiscrete groupoid.
The functor to the usual fundamental groupoid is neither full nor faithful in general.

The author thanks Ronnie Brown and the referee for pointing out the work of Brown--Janelidze and Kamps--Porter, and Ronnie Brown again for picking up an error regarding a claimed relation between the current paper and those papers.

\section{Semilocal 1-connected maps}

Now for the definition of the type of map in which we will be interested. 
The reader is invited to draw the analogy with what one might call `relativisation' of properties of schemes to properties of maps of schemes in algebraic geometry.

\begin{definition}

	Given a map $p\colon Y\to X$, we say $p$ is \emph{semilocally 1-connected}, or \emph{sl1c},  if for all $y \in Y$ there is a basic open neighbourhood $U$ of $p(y)$ and a basic open neighbourhood $N$ of $y$ with $p(N) \subset U$ such that for any square
	\[
		\xymatrix{
			\partial I^2 \ar[d] \ar[r] & N \ar[r] & U\times_X Y \ar[d]\\
			I^2 \ar[rr] \ar@{-->}[urr] && U
		}
	\]
	there is a lift as shown.

\end{definition}

\begin{example}

	A fibre bundle $E\to B$ with semilocally 1-connected fibre $F$ is an sl1c map.

\end{example}

We now give a pair of useful lemmas giving the behaviour of sl1c maps. 
The notion of a property of maps being \emph{local on the base} is familiar in algebraic geometry, but less so in algebraic topology, so we recall it here.

\begin{definition}

	A property $P$ of continuous maps is \emph{local on the base} if for any map $p\colon X\to Y$ and any open cover $\{U_i \to Y\}$, then $p$ has property $P$ if and only if each $U_i\times_Y X \to U_i$ has property $P$.

\end{definition}

\begin{lemma}
	\label{lemma:sl1c_local_on_base}

	The property of being sl1c is local on the base, and is preserved under arbitrary pullback.

\end{lemma}

The following lemma is the relative analogue of semilocal simple connectedness being preserved by products. 

\begin{lemma}

	If $X\to Z \leftarrow Y$ is a cospan of sl1c maps, the map  $X\times_Z Y \to Z$ is sl1c.

\end{lemma}

\section{Definition of $\pionevert$}

Fix a map $p\colon Y\to X$ of topological spaces. We let the set of objects of $\pionevert(p)$ be just (the underlying set of) $Y$. 
The set of arrows is the quotient $[I,Y]_{vert}$ of the set of vertical paths $I\to Y$ by the relation of vertical homotopy relative to $\{0,1\} \subset I$. 
Source and target maps are given by evaluation at $0,1\in I$ respectively; the unit map sends $y\in Y$ to the  constant path at $y$; inversion is induced by the involution $t\mapsto 1-t$ on $I$; composition is concatenation of paths. 
The standard proofs for associativity etc.\ follow as for the usual fundamental groupoid, so we will not dwell on them here. 
Note that there is a functor from $\pionevert(p)$ to the groupoid given by the set underlying $X$, which shall also be denoted $p$; it shall be clear from context whether the functor or the map is intended.

More interesting is the definition of the topology on $\pionevert(p)_1$ and the proof that the above operations are continuous. We assume that $[I,Y]_{vert}$ is given the subspace topology induced from the compact-open topology on $[I,Y]$.

Let $[f]\in \pionevert(p)_1$, $U \subset X$ a basic open neighbourhood containing $p[f]$ and $V_0,V_1$ basic open neighbourhoods of $s[f]$ and $t[f]$  respectively. 
The basic open neighbourhoods of $[f]$ are those sets of the form
\begin{align*}
	\langle[f],U,V_0,V_1\rangle:=\{[g] \in [I,Y]_{vert}\ \mid\  &\exists h								\colon I^2 \to p^{-1}(U), h(0,-) = f,\\
  								& h(1,-) = g, 
  								 h(\epsilon,t) \in V_\epsilon, 
  								 \epsilon=0,1\}
\end{align*}
One can check that in the case that $X=\ast$, this reduces to the definition of the topology on the arrows of $\Pi_1(Y)$ (see e.g.\ \cite{Brown-Danesh-Naruie_75}).

\begin{lemma}
	The map $[I,Y]_{vert} \to \pionevert(p)_1$ is continuous.
\end{lemma}

This allows us to define a continuous map to $\pionevert(p)_1$ by first defining a map to $[I,Y]_{vert}$.

In order to show that multiplication in the vertical fundamental groupoid is continuous we need the relative form of the usual semilocal simple connectedness condition.

\begin{lemma}

	Let $X$ be locally 1-connected and $p\colon Y\to X$ be semilocally 1-connected. 
	Then multiplication in the groupoid $\pionevert(p)$ is continuous.

\end{lemma}

\begin{proof}

	Let $[f],[g]$ be composable arrows in $\pionevert(p)$, and take a 
	basic open neighbourhood $\langle[fg],U,V_0,V_2\rangle$ of $[fg]$.
	Let $V_1$ be a basic open neighbourhood of $t[f]=s[g]$ such that $p(V_1)\subset U$. 
	Then let us consider the image of the basic open neighbourhood
	\[
		\langle[f],U,V_0,V_1\rangle \times_Y \langle[g],U,V_1,V_2\rangle
	\]
	under the multiplication map. 
	We will show this is contained in the set $\langle[f][g],U,V_0,V_2\rangle$

	Let $([k],[l])$ be an element of the above set. 
	There are homotopies $H_k$ and $H_l$ in $Y_U$ and we get a map $\partial I^2 \to V_1$ which is constant on the vertical edges, and given by $t\circ H_l$ and $s\circ H_k$ on the other two edges. 
	Since by assumption $U$ is simply connected, we can find an extension to $I^2$ of the map $\partial I^2 \to V_1 \to U$. As $p$ is semilocally 1-connected, we can lift this extension to a map $h\colon I^2 \to Y_U$.

	We then paste the homotopies $H_k$, $H_l$ and $h$ to get a homotopy between $[f][g]$ and $[k][l]$, so that $[k][l]\in \langle[f][g],U,V_0,V_2\rangle$.
	Thus
	\[
		m\left(\langle[f],U,V_0,V_1\rangle \times_Y \langle[g],U,V_1,V_2\rangle\right)
		\subset \langle[f][g],U,V_0,V_2\rangle,
	\]
	and hence the multiplication map is continuous.
\end{proof}

The continuity of the source, target, unit and inverse maps is easy to check (and in fact doesn't need the topological condition in the above lemma).

\begin{example}
	\label{eg:pionevert_of_projection}

	Consider the projection map $\pr_X\colon X\times F \to X$ where $X$ is locally 1-connected and $F$ is semilocally 1-connected. 
	Then there is an isomorphism of topological groupoids over $X$,
	\[
		\pionevert(\pr_X) \simeq X\times \Pi_1(F),
	\]
	where $\Pi_1(F)$ is given the usual topology. In particular, if $F = I$ the unit interval
	\[
		\pionevert(\pr_X) \simeq X\times \codisc(I),
	\]
	where $codisc(I)$ is the topological groupoid with object space $I$ and a unique arrow between any two points in $I$.

\end{example}

\begin{lemma}

	For semilocally 1-connected maps $p\colon Y\to X$ and $q\colon Z\to X$ there is a natural isomorphism of topological groupoids
	\[
		\pionevert(Y\times_X Z \to X) \simeq \pionevert(p)\times_X \pionevert(q)
	\]
	over $X$.

\end{lemma}

Let $p_i\colon Y_i \to X$, $i=1,2$, be sl1c maps over a locally 1-connected space $X$.
Given a map $Y_1 \to Y_2$ over $X$, there is clearly a functor
\[
	\pionevert(p_1) \to \pionevert(p_2)
\]
which can easily be seen to be continuous. 
Thus $\pionevert$ is functorial for maps in the slice category $\Top^h/_{sl1c}X$, the full subcategory of $\Top/X$ on the sl1c maps. 
Also, if we have a commuting square
\[
	\xymatrix{
		Y_1 \ar[r] \ar[d]_{q_1} & Y_2 \ar[d]^{q_2}\\
		X_1 \ar[r] & X_2
	}
\]
where the $X_i$ are locally 1-connected and the $p_i$ are sl1c, we have a commuting square of functors
\[
	\xymatrix{
		\pionevert(q_1) \ar[r] \ar[d]_{q_1} & \pionevert(q_2) \ar[d]^{q_2}\\
		X_1 \ar[r] & X_2
	}
\]

Returning to our fixed base space $X$, given a homotopy $H\colon I\times Y_1\to Y_2$ over $X$ we can use example \ref{eg:pionevert_of_projection} to get a functor
\[
	\mathbf{2}\times \pionevert(p_1) \hookrightarrow 
	\codisc(I)\times \pionevert(p_1)
	\simeq \pionevert(p_1\circ \pr_{Y_1})
	\to \pionevert(p_2)
\]
over $X$, where $\mathbf{2} = \codisc(\mathbf{2})$ is the groupoid with two objects and a unique isomorphism between them, and the inclusion into $\codisc(I)$ maps those two objects to the endpoints of $I$. 
This is nothing but a natural transformation between the two functors induced by the maps $H(0,-)$ and $H(1,-)$.  
In fact $\pionevert$ respects more structure than this, in that it is a \emph{simplicial functor} for simplicial enrichments arising from the usual ones on $\Top$ and $\Gpd$. 
We give for completeness the definitions of the simplicial enrichments we consider on the categories $\Top/_{sl1c}X$ and $\TopGpd/X$:

\begin{itemize}

	\item[--]
	The simplicial hom in $\Top/_{sl1c}X$ is given by
	\[
		[Y_1,Y_2]^{\Top/X}_n = \Top_{/X}(\Delta^n \times Y_1,Y_2),
	\]
	where $\Delta^n$ is the standard topological $n$-simplex.

	\item[--]
	The simplicial hom in $\TopGpd/X$ is given by
	\[
		[Z_1,Z_2]^{\TopGpd/X}_n = \TopGpd_{/X}(\codisc(\mathbf{n})\times Z_1,Z_2),
	\]
	where $\codisc(\mathbf{n})$ is the contractible groupoid with $n$ objects.

\end{itemize}

Both simplicial enrichments arise from a cosimplicial structure: in the first case the standard cosimplicial structure on $\Delta^\bullet$, in the second case the groupoids $\mathbf{n}$ likewise assemble into a cosimplicial groupoid $\codisc(\bullet)$.
The functor $\Pi_1$ sends the cosimplicial space $\Delta^\bullet$ to the cosimplicial groupoid $\Pi_1(\Delta^\bullet)$.
There is a functor $\codisc(\mathbf{n}) \to \Pi_1(\Delta^n)$ sending each object to a vertex giving a map of cosimplicial groupoids.

\begin{proposition}

	$\pionevert$ is a simplicial functor $\Top/_{sl1c}X \to \TopGpd/X$.

\end{proposition}

\begin{proof}

	We will describe the map on the hom-simplicial sets and show it respects composition; the rest is safely left to the reader. 
	Given an $n$-simplex in $[p_1,p_2]^{\Top/X}$, namely a map
	\[
		\Delta^n \times Y_1\to Y_2
	\]
	over $X$, send it to the $n$-simplex
	\begin{align*}
		\codisc(\mathbf{n})\times \pionevert(p_1) \hookrightarrow \codisc(\Delta^n)\times \pionevert(p_1)
	&\simeq \pionevert(p_1\circ \pr_{Y_1})\\
	& \to \pionevert(p_2)
	\end{align*}
	in $[\pionevert(p_1),\pionevert(p_2)]^{\TopGpd/X}$.
	This is a map of simplicial sets because $\codisc(\bullet) \to \Pi_1(\Delta^\bullet)$ is a cosimplicial map.

	Now given a pair of composable $n$-simplices in $\Top/_{sl1c}X$, say
	\begin{align*}
		f\colon &\Delta^n \times Y_1\to Y_2\\
		g\colon &\Delta^n \times Y_2\to Y_3
	\end{align*}
	their composite is the $n$-simplex
	\[
		\Delta^n \times Y_1 \xrightarrow{\Delta\times \id} \Delta^n\times \Delta^n \times Y_1 \xrightarrow{\id \times f} \Delta^n \times Y_2 \xrightarrow{g} Y_3.
	\]
	That composition is respected follows from the commutativity of the diagram
	\[
		\xymatrix{
			\codisc(\mathbf{n}) \ar[r]^\Delta \ar[d] & \codisc(\mathbf{n}\times \mathbf{n}) \ar[d]\\
			\codisc(\Delta^n) \ar[r]_-\Delta & \codisc(\Delta^n\times \Delta^n)
		}
	\]
	and canonical isomorphisms.
\end{proof}

Note that we also have the $\pionevert$ gives a functor from the arrow category of $\Top$ (or rather the full subcategory on the sl1c maps) to the arrow category of $\TopGpd$, preserving cartesian squares (as shown by the following proposition).

\begin{proposition}

	For a map $f\colon W\to X$ with both $W$ and $X$ locally 1-connected, and a semilocally 1-connected map $p\colon Y\to X$ there is a natural isomorphism
	\[
		\pionevert(f^*(p)) \stackrel{\simeq}{\to} 
			\disc(W)\times_X \pionevert(p).
	\]

\end{proposition}

More concretely, it means we can pull back vertical fundamental groupoids along any map. In particular one can restrict to open subsets of the codomain.

\section{Applications}

In the section below we use a coarser notion of equivalence of topological groupoids, inspired by the theory of stacks, and which has appeared in many different guises (see for instance \S 8 of \cite{Roberts_12} for history and examples).

\begin{definition}

	A functor $f\colon D\to E$ between topological groupoids is a \emph{weak equivalence} if $D_1 \simeq D_0^2 \times_{E_0^2} E_1$ (full faithfulness) and 
	\[
		t\circ pr_2\colon D_0\times_{E_0} E_1 \to E_0
	\]
	admits local sections (essential surjectivity). 
	A pair of groupoids are \emph{weakly equivalent} if there is a span of weak equivalences between them.

\end{definition}

An internal equivalence of topological groupoids, that is a pair of functors which are quasi-inverse, gives an example of a weak equivalence. Another example is the induced functor $\check{C}(V) \to \check{C}(U)$ arising from a map $V\to U$ of open covers of a space $X$. This latter includes the case when $U=X$, so that $\check{C}(V) \to X$ is a weak equivalence.

\begin{example}
	
	Let $X$ be a semilocally 1-connected space. 
	Then the functor
	\[
		\Pi_1(X)^{disc} \to \Pi_1(X)^{top}
 	\]
 	is a weak equivalence. 
 	Here the superscripts indicate the discrete topology and the usual topology as given in \cite{Brown-Danesh-Naruie_75}.
 	This functor is bijective on objects and arrows, but is not an equivalence except in degenerate situations.

\end{example}

\subsection{Fibrewise universal covering spaces}

The universal covering space $\widetilde{X} \to X$ of a connected, semilocally 1-connected space $X$ is isomorphic to the source fibre $s^{-1}(x)$ of the fundamental groupoid $\Pi_1(X)$ at an object $x\in X$. 
In the fibrewise setting, with all fibres connected, this object becomes a section of the sl1c map at hand. 
Borrowing the language of topos theory, one might call this a \emph{global} point in the category of spaces over $X$. 
A map posessing a section is called an \emph{ex-space} in the parameterised homotopy theory literature (see e.g.\ \cite{Crabb-James_98}).
The reader will of course immediately appreciate that not all maps of interest come with sections. 
However, different basepoints in a connected space give isomorphic universal covering spaces, and the existence of enough \emph{local} sections will be enough to define a fibrewise universal covering space. 
The author does not know whether more general input will also yield a fibrewise universal covering space, but expects it may well be possible.

Fix an ex-space $p\colon Y\to X$ with section $s\colon Y\to X$ such that $p$ is an sl1c map and $X$ is locally 1-connected.
We can form the groupoid $\disc(X)\times_{\disc(Y)}\pionevert(p)$ with objects $X$. 
This has source and target maps	equal, so gives a bundle of groups over $X$.
The fibre over $x\in X$ is the fundamental group of $Y_x$ at the point $s(x)$, and one can think of this bundle as a discrete group object in spaces over $X$.
If all the fibres of $p$ are path connected, then this captures the homotopy 1-types of the fibres. 
In this case, the fibrewise universal covering space, defined below, will be a principal bundle in the category of spaces over $X$ with structure group this relative $\pi_1$.

We will consider not just ex-spaces, but spaces with a collection of local sections that among themselves satisfy a homotopy coherence condition.

\begin{definition}

	A \emph{coherent local splitting} of a map $p\colon Y \to X$ is an open cover $U = \coprod_i U_i \to X$ and a collection of local sections $s_i\colon U_i \to Y$ together with vertical homotopies $h_{ij}$ from $s_i\big|_{U_{ij}}$ to $s_j\big|_{U_{ij}}$ that satisfy the cocycle equation
	\[
		h_{ij} + h_{ij} \simeq h_{ik}
	\]
	up to a vertical homotopy (the sum denotes the composite homotopy). It is also assumed that $h_{ii}$ is the constant homotopy. A \emph{locally ex-space} is a map with a specified coherent local splitting.

\end{definition}

There are many examples such locally ex-spaces, arising from certain fibrations. Given $Y_i \to X$, $i=1,2$, a \emph{fibre homotopy equivalence} between $Y_1$ and $Y_2$ is a pair of maps $f\colon Y_1 \leftrightarrows Y_2 \cocolon g$ over $X$, such that $f\circ g$ and $g\circ f$ are homotopic over $X$ to identity maps. 

\begin{definition}

	A \emph{locally homotopy trivial fibration} is a fibration $p\colon E\to B$ such that there is an open cover $\{U_i \to B\}$, and fibre homotopy equivalences between $E\big|_{U_i}$ and $U_i\times F$ for each $i$, for some space $F$.

\end{definition}

\begin{example}
	
	A locally homotopy trivial fibration $p\colon E \to B$, say with path-connected fibres, can be given the structure of a locally ex-space.

\end{example}

In this section we shall focus on locally ex-spaces that have underlying sl1c maps. 
These are, in the author's opinion, the `correct' fibrewise generalisation of a pointed semilocally 1-connected space.
We shall also assume that the codomains of all sl1c maps are locally 1-connected.

As a point of notation, recall that given an open cover $U = \coprod_i U_i \to M$ of a space $M$, there is a topological groupoid $\check{C}(U) \to M$ over $M$, called the \emph{\v Cech groupoid} of the cover.

\begin{lemma}
	
	Given an sl1c locally ex-space $Y \to X$, with open cover $U\to X$, and with $X$ locally 1-connected, there is a continuous functor $s\colon \check{C}(U) \to \pionevert(p)$ such that $p\circ s$ is the canonical projection $\check{C}(U) \to X$.

\end{lemma}

\begin{proof}
	
	Since the map $[I,Y]_{vert} \to \pionevert(p)_1$ is continuous, and we are given the homotopies $U_i\times_X U_j \to [I,Y]_{vert}$, the obvious maps $U_i\times_X U_j \to \pionevert(p)_1$ are continuous; these then give the arrow components of the functor $\check{C}(U) \to \pionevert(p)$. 
	That composition is preserved is due to the up-to-homotopy cocycle equation. 
	Preservation of identities follows from the assumption on $h_{ii}$.
\end{proof}

The putative fibrewise universal covering space can be defined as the space of orbits of a certain equivalence relation; a groupoid with at most one arrow between any two objects. Recall that for a (topological) groupoid $C$, the groupoid $C^\mathbf{2}$ has as objects the arrows of $C$, and as arrows commutative squares. There are two functors $S,T\colon C^\mathbf{2} \to C$ sending an arrow to its source and target respectively. The equivalence relation is in itself important for proving properties, so we shall proceed in several steps.

Fix a functor $s\colon\check{C}(U) \to \pionevert(p)$. First, define the iterated pullback
\[
	\xymatrix{
		\widehat{Y}^{(1),s} \ar[rr] \ar[dd] && Y \ar[d]\\
		& \pionevert(p)^\mathbf{2} \ar[r]_T \ar[d]^S & \pionevert(p)\\
		\check{C}(U) \ar[r]_-{s} & \pionevert(p)
	}\;.
\]
This is a topological groupoid with a functor to $Y$. The objects are pairs consisting of homotopy classes of vertical paths $\gamma$ in $Y$ together with an $x\in U$ such that $\gamma(0) = s(x)$. The morphisms $(\gamma_1,x_1) \to (\gamma_2,x_2)$ are commutative triangles in $\pionevert(p)$ with two sides $\gamma_1$ and $\gamma_2$, together with a lift of the third side, $\gamma_1\gamma_2^{-1}$, to a map $x_1 \to x_2$ in $\check{C}(U)$. Moreover, this lift, if it exists, is unique. Thus there is at most one arrow between any two objects in $\widehat{Y}^{(1),s}$, making it an equivalence relation.

Secondly, if we take a open cover $i\colon V\to U$ of the open cover $U\to X$, we get a new coherent local splitting by precomposing with $i$, and so a new functor $s' = s\circ i$.

\begin{lemma}
	
	The canonical functor $\widehat{Y}^{(1),s'} \to \widehat{Y}^{(1),s}$ over $Y$ is a weak equivalence. 

\end{lemma}

The proof follows from the fact $\check{C}(V) \to \check{C}(U)$ is a weak equivalence with an open cover as object component, and such weak equivalences pull back to weak equivalences. 
Also, if we take a natural isomorphism $s_1 \Rightarrow s_2 \colon\check{C}(U) \to \pionevert(p)$, then there is a canonical functor $\widehat{Y}^{(1),s_1} \to \widehat{Y}^{(1),s_2}$ over $Y$, and this functor is a weak equivalence.

\begin{lemma}

	Given a weak equivalence $A\to B$ of topological groupoids over $Y$, there is an induced isomorphism $A_0/A_1 \to B_0/B_1$ of spaces of orbits, again commuting with the maps to $Y$. 

\end{lemma}

We will define a locally ex-space $p$ to be \emph{connected} if for any two coherent local splittings there is a natural transformation between functors out of a common refinement of the two \v Cech groupoids involved. 
This implies that the fibres of $p$ are path connected, but the reverse implication does not necessarily hold without further assumptions on $p$.

\begin{corollary}

	For a connected sl1c locally ex-space the space of orbits of $\widehat{Y}^{(1),s}$ is, up to a isomorphism over $Y$, independent of the choice of coherent local splitting $s$.

\end{corollary}

This corollary is analogous to the result that the universal covering space of a connected space is independent of the choice of basepoint used to construct it.

\begin{definition}

	A \emph{fibrewise universal covering space} $Y^{(1),s}$ of a connected sl1c locally ex-space with coherent local splittings is the space of orbits of any of the groupoids $\widehat{Y}^{(1),s}$ for any choice $s$ of coherent local splitting.

\end{definition}

The usual abuse of notation will apply, and we shall denote any fibrewise universal covering space simply by $Y^{(1)}$.

\begin{proposition}
	
	For any $x\in X$, the restriction of $Y^{(1)}$ to a fibre $Y_x = p^{-1}(x) \subset Y$ is a universal covering space $\widetilde{Y_x}$

\end{proposition}

\begin{proof}

	One can take a coherent local splitting $s$ such that the point $x$ is only contained in a single open set in the cover of $X$.
	Then $\widehat{Y}^{(1),s}\times_Y Y_x$ is isomorphic to the source fibre of $\Pi_1(Y_x)$ at $s(x)\in Y$, but this is a universal covering space of $Y_x$.
\end{proof}

Finally, let us consider the case when the locally ex-space is not connected.
When one wants to define a universal covering space of a non-connected space $X$, the construction goes through as long as instead of taking a basepoint, the input is instead a section $\pi_0(X) \to X$ of the natural projection to path components.
In the fibrewise case $\pi_0$ is replaced by the bundle $\pi_0(\pionevert(p)) =:\pi_0^{vert}(p) \to X$ of vertical path components of $p\colon Y\to X$. 
There is a map $Y \to \pi_0^{vert}(p)$ of spaces over $X$, so $Y$ becomes a space over $\pi_0^{vert}(p)$.
The coherent local splitting of $p$, analogous to a single basepoint, now gets replaced by a coherent local splitting of $Y \to \pi_0^{vert}(p)$. 
The theory then goes through as before.
Note that in the case that $p$ is a connected locally ex-space, $\pi_0^{vert}(p)=X$, so we recover the above treatment for connected $p$.

\subsection{2-bundles}

A topological 2-bundle (see e.g.\ \cite{Bartels}) is a categorified (principal or fibre) bundle, where the fibres are replaced by topological groupoids. 
One still asks for local triviality, but in a sense that is a little subtle. 
Using the terminology of anafunctors from \cite{Bartels,Roberts_12}, one can say that a 2-bundle is locally equivalent to a product, where the equivalence is an anafunctor rather than a functor.
We shall give a simpler description that is good enough for present purposes.

\begin{definition}

	A \emph{2-bundle} over a space $X$ is a topological groupoid $Z$, equipped with a functor to $X$ (considered as a trivial topological groupoid), that is locally over $X$ weakly equivalent to a product, i.e.\ a groupoid of the form $F\times U$, for some topological groupoid $F$.

\end{definition}

It should be clear that a fibre homotopy equivalence between $p_1\colon Y_1 \to X$ and $p_2\colon Y_2 \to X$ gives rise to an internal equivalence 
\[
    \pionevert(p_1) \leftrightarrows \pionevert(p_2)
\]
of topological groupoids over $X$, in the sense that the two back and forth composites are (continuously) naturally isomorphic to the identity functors. 
In particular, if $Y_2 = X\times F$ and $p_2$ is projection onto $X$, then we have an internal equivalence
\[
    \pionevert(p_1) \leftrightarrows X\times \Pi_1(F)
\]
Note that an internal equivalence is \emph{a fortiori} a weak equivalence.

Thus we have the following result.

\begin{proposition}

	Assume the space $B$ is locally 1-connected. Given a locally homotopy trivial fibration $p\colon E\to B$ with semilocally 1-connected typical fibre $F$, $\pionevert(p)$ is a 2-bundle with typical fibre $\Pi_1(F)$.

\end{proposition}

In fact one can say a little more, in that since $\Pi_1(F)$ is weakly equivalent to a topologically discrete groupoid, $\pionevert(p)$ is a \emph{2-covering space} in the sense of \cite{Roberts_phd}. 
It only remains to give a ready supply of examples of locally homotopy trivial fibrations in order to give plenty of 2-bundles. 
But any Hurewicz fibration over a space with a cover by open sets with null-homotopic inclusions is a locally homotopy trivial fibration \cite{Dold_63}. 
In fact any Dold fibration (a weaker notion) over such a space is locally homotopy trivial.

Given a locally relatively contractible space $X$, the path fibration $\ev_1\colon PX \to X$ is a locally homotopy trivial fibration, and the 2-bundle $\pionevert(\ev_1)$ in that case is the canonical 2-covering space from \cite{Roberts_phd}.

In a slightly different vein, recall that given a (semilocally 1-connected) topological group $G$, the fundamental groupoid is a strict 2-group (in the sense of a categorical group, not a $p$-group for $p=2$). 
A strict 2-group is the same thing as a crossed module, and in this case is the crossed module $\widetilde{G_0} \to G$ where $G_0$ is the connected component of the identity.
Given a bundle $\pi\colon G \to X$ of topological groups over a locally 1-connected space (with projection map sl1c), we can define the bundle of strict 2-groups $\pionevert(\pi) \to X$.
This then gives rise to a bundle of crossed modules over $X$, where we take the identity section of $\pi$ to define the fibrewise universal covering space of the fibrewise identity component of $G$.

\subsection{A topological fundamental bigroupoid}

Now we shall restrict attention to the map $E:=(\ev_0,\ev_1)\colon X^I \to X\times X$, and so we need to see when this is sl1c. 
In order to study the local properties of the mapping space $X^I$, we shall give a specific open neighbourhood basis for it.

Let $\gamma \in X^I$, and let $\mathfrak{p} = \{t_1,\ldots,t_{2n}\mid t_i \in I,\ t_1 < \ldots < t_{2n}\}$ be a partition of the unit interval. 
Also, let $W=\{W_i\}_{i=1}^{2n+1}$ be a collection of basic open neighbourhoods of $X$ such that

\begin{itemize}
	\item
		$\gamma([t_{i-1},t_i]) \subset W_i$ (by convention, let us take $t_0 = 0,\ t_{2n+1} = 1$), and

	\item
		$W_{2i} \subset W_{2i-1}\cap W_{2i+1}$ for $i=1,\ldots,n$.

\end{itemize}
Then define the set $N(\mathfrak{p},W)\subset X^I$ to consist of those paths $\eta$ such that $\eta([t_{i-1},t_i]) \subset W_i$.

\begin{lemma}
	\label{lemma:basic_opens_path_space}

	The sets $N(\mathfrak{p},W)$ form a collection of basic open neighbourhoods for the compact-open topology on $X^I$.
	
\end{lemma}

Recall that the space $X$ is \emph{semilocally 2-connected} if it has an open neighbourhood basis consisting of 1-connected sets $U$, the inclusion of which induces the zero map $\pi_2(U) \to \pi_2(X)$.

\begin{lemma}

	The space $X$ is semilocally 2-connected if and only if $E$ is sl1c with locally 1-connected codomain.

\end{lemma}

\begin{proof}

	Firstly, note that $X^2$ has an open neighbourhood basis consisting of 1-connected sets if and only if $X$ has an open neigbourhood basis of 1-connected sets. 
	Assume $(ev_0,ev_1)$ is sl1c and let $f\colon \partial I^3 \to U$ be a map to a basic open nhd $U \subset X$. 
	We can view this as a loop in $X^I$ with projection down to $X^2$ a constant loop at $(f(N),f(S))$, where $N$ and $S$ are the north and south poles of $S^2\simeq \partial I^3$. 
	We have the obvious  filler for this loop, and we can lift this to a filler for the loop in $X^I$ in $(U^2)\times_{X^2} X^I$, giving a filler for the map from the sphere in $X$. 
	Thus $X$ is semilocally 2-connected.

	Now if $X$ is semilocally 2-connected, a basic open neighbourhood of $\gamma \in X^I$ is as in lemma \ref{lemma:basic_opens_path_space}, where the open sets $W_i$ are 1-connected. 
	So we take as basic open neighbourhood of $(\gamma(0),\gamma(1))$ the set $W_0\times W_{2n+1}$.  
	
	Then given the lifting problem as in the definition of an sl1c map, a lift is given by choosing inductively fillers for the cylinders $I\times \partial I^2 \to W_i$ and the spheres $I^2 \cup I^2 \to W_i \cap W_{i+1}$. 
	Since $W_i \cap W_{i+1}$ is contained in either of $W_i$ or $W_{i+1}$ is is possible to fill after inclusion into $X$. 
	Thus $(ev_0,ev_1)$ is sl1c, and we are done.
\end{proof}

Bigroupoids are one model for the homotopy 2-types of spaces -- though there are other, strict models, which we shall not touch on here. 
The definition of topological bigroupoid presented here is in the style of a weak enrichment by topological groupoids, but now the hom-groupoids needs to vary continuously over the space of objects.
This is a little more concise than the definitions in \cite{Danny_phd,HKK_01}, which write out all the components explicitly.

\begin{definition}\label{defn:top_bigpd}

	A \emph{topological bigroupoid} $B$ is a topological groupoid $\underline{B}_1$ equipped with a functor $(S,T)\colon  \underline{B}_1 \to \disc(B_0\times B_0)$ for a space $B_0$ together with:

  \begin{itemize}

    \item[--]

      Functors

      \begin{align*}
      	C\colon & \underline{B}_1 \times_{S,\disc(B_0),T}  \underline{B}_1  \to  \underline{B}_1\\
      	I\colon &\disc(B_0) \to  \underline{B}_1
      \end{align*}

      over $\disc(B_0\times B_0)$ and a functor 
      \[
        \overline{(-)}\colon  \underline{B}_1 \to  \underline{B}_1
      \]
      covering the swap map $\disc(B_0\times B_0) \to \disc(B_0\times B_0)$.

    \item[--]

      Continuous maps

      \[
          \begin{array}{l}
            a\colon\Obj( \underline{B}_1) \times_{S,B_0,T} \Obj( \underline{B}_1)\times_{S,B_0,T} 
            \Obj( \underline{B}_1) \to \Mor( \underline{B}_1)   \\
            r\colon\Obj( \underline{B}_1) \to \Mor( \underline{B}_1)\\
            l\colon\Obj( \underline{B}_1) \to \Mor( \underline{B}_1)\\
            e\colon\Obj( \underline{B}_1) \to \Mor( \underline{B}_1)\\
            i\colon\Obj( \underline{B}_1) \to \Mor( \underline{B}_1)
          \end{array}
      \]

      which are the component maps of natural isomorphisms
      \[
      	\xymatrix{
      		 \underline{B}_1 \times_{S,\disc(B_0),T}  \underline{B}_1\times_{S,\disc(B_0),T}  
      		 \underline{B}_1
      		\ar[rr]^(.65){\id\times C} \ar[dd]_{C\times \id} 
      		&&
      		 \underline{B}_1 \times_{S,\disc(B_0),T}  \underline{B}_1
      		\ar[dd]_{\ }="t"^C\\
      		\\
      		 \underline{B}_1 \times_{S,\disc(B_0),T}  \underline{B}_1
      		\ar[rr]^(.8){\ }="s"_C
      		&&
      		 \underline{B}_1
      		\ar@{=>}"s";"t"^a
      	}
      \]\\
      \[
      	\xymatrix@C-1pc{
      		 \underline{B}_1 \times_{S,\disc(B_0)} \disc(B_0)
      		\ar[rr]^{\id\times I} \ar[rrdd]^{\ }="t1"_\simeq && 
      		 \underline{B}_1 \times_{S,\disc(B_0),T}  \underline{B}_1
      		\ar[dd]^(0.46)C="s2"_(0.46){\ }="s1" && \disc(B_0) \times_{\disc(B_0),T}  \underline{B}_1
      		\ar[ll]_{I\times\id} \ar[ddll]_{\ }="t2"^\simeq\\
      		\\
      		&&  \underline{B}_1
      		\ar@{=>}"s1";"t1"_r
      		\ar@{=>}"s2";"t2"^l
      	}	
      \]\\
      \[
      	\xymatrix{
      		 \underline{B}_1 \ar[rr]^{(\overline{(-)},\id)} \ar[dd]_S && 
      		 \underline{B}_1 \times_{S,\disc(B_0),T}  \underline{B}_1
      		\ar[dd]_{\ }="s1"^C="t2" &&
      		 \underline{B}_1 \ar[ll]_{(\id,\overline{(-)})} \ar[dd]^T\\
      		\\
      		\disc(B_0)\ar[rr]^{\ }="t1"_I &&  \underline{B}_1 && \disc(B_0)\ar[ll]_{\ }="s2"^I
      		\ar@{=>}"s1";"t1"_e
      		\ar@{=>}"s2";"t2"_i
      	}
      \]

	\end{itemize}

	The following diagrams are required to commute:
	\[
	  	\xymatrix@C-2pc{
	  		&& \ar[ddll]_{a_{(k\circ h)gf}} ((k\circ h)\circ g)\circ f \ar[rr]^{a_{khg}\circ \id_f} && (k\circ(h\circ g))\circ f \ar[ddrr]^{a_{k(h\circ g)f}}\\
	  		\\
	  		(k\circ h)\circ(g\circ f) \ar[ddrrr]_{a_{kh(g\circ f)}} &&& &&& k\circ((h\circ g)\circ f) \ar[ddlll]^{\id_k\circ a_{hgf}}\\
	  		\\
	  		&&&k\circ(h\circ(g\circ f))
	  	}
	\]
  
	\[
	  	\xymatrix{
	  		(g\circ I)\circ f \ar[rr]^{a_{gIf}} \ar[dr]_{r_g \id_f}&& g\circ(I\circ f) \ar[dl]^{\id_g l_f}\\
	  		&g\circ f
	  	}
	\]

	\[
	  	\xymatrix{
	  		I\circ f \ar[rr]^(.35){i_f\circ \id_f} \ar[dd]_{l_f} && (f\circ\overline{f})\circ f \ar[rr]^{a_{f\overline{f}f}} &&
	  		f\circ(\overline{f} \circ f) \ar[dd]^{\id_f\circ e_f}\\
	  		\\
	  		f \ar[rrrr]_{r_f^{-1}} &&&&f\circ I
	  	}
	\]

\end{definition}

\begin{proposition}

	Given a semilocally 2-connected space $X$, applying $\pionevert$ to $E\colon X^I \to X^2$ gives a topological bigroupoid $\Pi_2(X)$.

\end{proposition}

\begin{proof}

	We first need to specify the structure maps.
	The source and target functor $(S,T)\colon \pionevert(E) \to X^2$ is automatic, as is the functor $\disc(X) \to \pionevert(E)$, by applying $\pionevert$ to the map $X \to X^I$ of spaces over $X^2$. 
	As $\Delta\colon X\to X^2$ is injective, $\pionevert(\Delta) \simeq X$, the groupoid with only identity arrows.
	Again, the reverse map $X^I \to X^I$ covers the swap map, and so induces a functor $\pionevert(E)\to \pionevert(E)$ as needed.
	The concatenation map gives the composition $\pionevert(E)\times_X\pionevert(E) \to \pionevert(E)$.

	The natural transformations $a,l,r,e,i$ are given by applying $\pionevert$ to the usual (vertical over $X^2$) homotopies that show that path composition is associative up to homotopy, constant paths are units up to homotopy and so on.
	Since any two vertical homotopies between the different concatenation orders of four paths are themselves vertical homotopic, we see that the pentagon equation must hold, and likewise for the other coherence diagrams.
\end{proof}

Note that the proof uses the structure of the interval as an $A_\infty$-cogroupoid, and so the structure of the path space as an $A_\infty$-groupoid, or indeed even as an algebra for the \emph{partial little intervals operad} \cite{Caruso-Waner_81} (see \cite{Cheng_CT14} for related discussion).
This is important in Trimble's definition of a flabby $n$-groupoid, and we can see here that we are essentially picking out just the lowest two dimensions of the structure.
Essentially we are truncating the hom $\infty$-groupoid of $\Pi_\infty(X)$ to a groupoid, but keeping the weak enrichment, so getting a bigroupoid.
One is tempted to conjecture that higher-dimensional analogues of this result exist, namely, given a semilocally $n$-connected space $X$, there is a topological fibrewise fundamental Trimble $(n-1)$-groupoid of the map $X^I \to X^2$, and that this is a topological fundamental $n$-groupoid of $X$.
Given connected $X$, the source fibre of this $n$-groupoid at a point is then conjecturally a topological $(n-1)$-stack that is a universal $n$-connected cover of $X$.


\begin{thebibliography}{HKK01}

\bibitem[Bak07]{Bakovic_phd}
I.~Bakovic, \emph{Bigroupoid 2-torsors}, Ph.D. thesis,
  Ludwig-Maxmillians-Universit{\"a}t M{\"u}nchen, 2007, available from
  \url{http://www.irb.hr/korisnici/ibakovic/}.

\bibitem[Bar06]{Bartels}
T.~Bartels, \emph{{Higher gauge theory I: 2-Bundles}}, Ph.D. thesis, University
  of California Riverside, 2006, {arXiv:math.CT/0410328}.

\bibitem[BDN75]{Brown-Danesh-Naruie_75}
R.~Brown and G.~Danesh-Naruie, \emph{The fundamental groupoid as a topological
  groupoid}, Proc. Edinburgh Math. Soc. (2) \textbf{19} (1974/75), 237--244.

\bibitem[BJ04]{Brown-Janelidze}
R.~Brown and G.~Janelidze, \emph{Galois theory and a new homotopy double
  groupoid of a map of spaces}, Applied Categorical Structures \textbf{12}
  (2004), 63--80, {arXiv}:math/0208211.

\bibitem[Che14]{Cheng_CT14}
E.~Cheng, \emph{Iterating path spaces}, Talk at {C}ategory {T}heory 2014,
  {C}ambridge, 2014.

\bibitem[CJ98]{Crabb-James_98}
M.~Crabb and I.~James, \emph{Fibrewise homotopy theory}, Springer Monographs in
  Mathematics, Springer-Verlag, 1998.

\bibitem[CW81]{Caruso-Waner_81}
J.~Caruso and S.~Waner, \emph{An approximation to {$\Omega^n\Sigma^n$}}, Trans.
  Amer. Math. Soc. \textbf{265} (1981), no.~1, 147--162.

\bibitem[DN70]{Danesh-Naruie}
G.~Danesh-Naruie, \emph{Topological groupoids}, Ph.D. thesis, Southampton
  University, 1970.

\bibitem[Dol63]{Dold_63}
A.~Dold, \emph{Partitions of unity in the theory of fibrations}, Ann. Math.
  \textbf{78} (1963), no.~2, 223--255.

\bibitem[HKK01]{HKK_01}
K.~A. Hardie, K.~H. Kamps, and R.~W. Kieboom, \emph{A homotopy bigroupoid of a
  topological space}, Appl. Categ. Structures \textbf{9} (2001), 311--327.

\bibitem[Jam69]{James_69}
I.~M. James, \emph{Bundles with special structure. {I}}, Ann. of Math. (2)
  \textbf{89} (1969), 359--390.

\bibitem[KP99]{Kamps-Porter}
K.~H. Kamps and T.~Porter, \emph{A homotopy 2-groupoid from a fibration},
  Homology, Homotopy and Applications \textbf{1} (1999), 79--93.

\bibitem[MS06]{May-Sigurdsson}
J.~P. May and J.~Sigurdsson, \emph{Parametrized homotopy theory}, Mathematical
  Surveys and Monographs, vol. 132, American Mathematical Society, Providence,
  RI, 2006.

\bibitem[NSS14]{NSS_14}
T.~Nikolaus, U.~Schreiber, and D.~Stevenson, \emph{Principal $\infty$-bundles:
  general theory}, J. Homotopy Rel. Struct. (2014).

\bibitem[Rob09]{Roberts_phd}
D.~M. Roberts, \emph{Fundamental bigroupoids and 2-covering spaces}, Ph.D.
  thesis, University of Adelaide, 2009, Available from
  \url{http://hdl.handle.net/2440/62680}.

\bibitem[Rob12]{Roberts_12}
\bysame, \emph{Internal categories, anafunctors and localisation}, Theory Appl.
  Categ. \textbf{26} (2012), no.~29, 788--829, {arXiv:1101.2363}.

\bibitem[Rob13]{Roberts_13b}
\bysame, \emph{A bigroupoid's topology}, 2013, {arXiv:1302.7019}.

\bibitem[Ste00]{Danny_phd}
D.~Stevenson, \emph{The geometry of bundle gerbes}, Ph.D. thesis, Adelaide
  University, Department of Pure Mathematics, 2000, {arXiv:math.DG/0004117}.

\end{thebibliography}

\providecommand{\bysame}{\leavevmode\hbox to3em{\hrulefill}\thinspace}
\providecommand{\MR}{\relax\ifhmode\unskip\space\fi MR }
\providecommand{\MRhref}[2]{%
  \href{http://www.ams.org/mathscinet-getitem?mr=#1}{#2}
}
\providecommand{\href}[2]{#2}

\end{document}